\documentclass[reqno]{amsart}

\usepackage[latin1]{inputenc}
\usepackage{amsmath}
\usepackage{amsfonts}
\usepackage{amssymb}
\usepackage{fullpage}
\usepackage{hyperref}
\usepackage{enumitem}
\usepackage{todonotes}
\usepackage{cleveref}
\usepackage{breqn}
\setkeys{breqn}{breakdepth={1}}

\theoremstyle{plain}
\newtheorem{theorem}{Theorem}

\newtheorem{lemma}[theorem]{Lemma}

\theoremstyle{definition}

\theoremstyle{remark}

\begin{document}
    
    \author{Hieu D. Nguyen}
%    \email{nguyen@rowan.edu}
    \title{A New Proof of the Prouhet-Tarry-Escott Problem}
    \date{11-22-2014}

    \address{Department of Mathematics, Rowan University, Glassboro, NJ 08028.}
    \email{nguyen@rowan.edu}
    
    \subjclass[2010]{Primary 11}
    %\thanks{}
    
    \maketitle

    \begin{abstract}
        The famous Prouhet-Tarry-Escott problem seeks collections of mutually disjoint sets of non-negative integers having equal sums of like powers.  In this paper we present a new proof of the solution to this problem by deriving a generalization of the product generating function formula for the classical Prouhet-Thue-Morse sequence.
    \end{abstract} 

\section{Introduction}

The well-known Prouhet-Tarry-Escott (PTE) problem (\cite{DB},\cite{W2}) seeks $p\geq 2$ sets of non-negative integers $S_0$, $S_1$, ..., $S_{p-1}$ that have equal sums of like powers (ESP) up to degree $M\geq 1$, i.e.
\[
\sum_{n\in S_0} n^m =\sum_{n\in S_1} n^m= ...=\sum_{n\in S_{p-1}} n^m 
\]
for all $m=0,1,...,M$.  In 1851, E. Prouhet \cite{P} gave a solution, but without proof, by partitioning the first $p^{M+1}$ non-negative integers into the sets $S_0,S_1,...,S_{p-1}$ according to the assignment
\[
n\in S_{v_p(n)}
\]
Here, $v_p(n)$ is the generalized Prouhet-Thue-Morse sequence defined by computing the residue of  the sum of digits of $n$ (base $p$):
\[
v_p(n)=\sum_{j=0}^d n_j \mod p
\]
where $n=n_dp^d+...+n_0p^0$ is the base-$p$ expansion of $n$.  When $p=2$, $v(n):=v_2(n)$ generates the classical Prouhet-Thue-Morse sequence: 0,1,1,0,1,0,0,1,.... For example, the two sets
\begin{align*}
S_0=\{0,3, 5, 6, 9,10,12,15\} \\
S_1=\{1,2,4,7,8,11,13,14\}
\end{align*}
defined by $n\in S_{v(n)}$ solves the PTE problem with $p=2$ and $M=3$ since

\[
\begin{array}{lcl}
8 & = & 0^0+3^0+5^0+6^0+9^0+10^0+12^0+15^0 \\
& = & 1^0+2^0+4^0+7^0+8^0+11^0+13^0+14^0 \\ 
\\
60 & = & 0+3+5+6+9+10+12+15 \\
& = & 1+2+4+7+8+11+13+14 \\
\\
620 & = & 0^2+3^2+5^2+6^2+9^2+10^2+12^2+15^2 \\
& = & 1^2+2^2+4^2+7^2+8^2+11^2+13^2+14^2 \\ 
\\
7200 & = & 0^3+3^3+5^3+6^3+9^3+10^3+12^3+15^3 \\
& = & 1^3+2^3+4^3+7^3+8^3+11^3+13^3+14^3
\end{array}
\]
where we define $0^0=1$.

The first published proof of Prouhet's solution was given by D. H. Lehmer \cite{L} who in fact presented a more general construction of ESPs beyond those described by Prouhet's solution.  This was achieved by considering products of polynomials whose coefficients are roots of unity.  In particular, Lehmer defined
\begin{equation}\label{eq:Lehmer-pgf}
F(\theta)=\prod_{m=0}^{M+1} (1+\omega e^{\mu_m \theta}+\omega^2 e^{2\mu_m \theta}+...+\omega^{p-1} e^{(p-1)\mu_m \theta})
\end{equation}
where $\omega$ is a $p$-th root of unity and $\{\mu_0,...,\mu_M\}$ are arbitrary positive integers.  It is clear that $F(x)$ has a zero at $x=0$ of order $M+1$ so that its derivative vanishes up to order $M$, i.e. $F^{(m)}(0)=0$ for $m=0,1,...,M$.  On the other hand, Lehmer expanded $F(x)$ to obtain
\begin{equation} \label{eq:Lehmer-pgf-2}
F(\theta)=\sum_{a_0,...,a_{M}} \omega^{a_0+...+a_{M}}e^{(a_1\mu_0+...+a_M\mu_{M})\theta}
\end{equation}
where $a_0,...,a_{M}$ take on all integers from 0 to $p-1$.  Since 
\[
F^{(m)}(0)=\sum_{a_0,...,a_M} \omega^{a_0+...+a_{M}} (a_0\mu_0+...+a_M\mu_M)^m
\]
Lehmer was able to prove using linear algebra that
\[
\sum_{n\in S_0} n^m =\sum_{n\in S_1} n^m= ...=\sum_{n\in S_{p-1}} n^m 
\]
where he assigned $n=a_0\mu_0+...+a_M\mu_M \in S_k$ if $a_0+...+a_M = k$ mod $p$.  This solves the PTE problem by setting $\mu_m=p^m$ for all $m=0,1,...,M$.  Other proofs of Prouhet's solution have been given by E. M. Wright \cite{W} using multinomial expansion and by J. B. Roberts \cite{R} using difference operators  (see also \cite{W2}).    

Observe that in the case mentioned where $\mu_m=p^m$ for all $m=0,...,M$, then equating (\ref{eq:Lehmer-pgf}) with (\ref{eq:Lehmer-pgf-2}) together with the substitution $x=e^{\theta}$ yields the product generating function formula 
\begin{equation}\label{eq:pgf-root-of-unity}
\prod_{m=0}^{M+1} (1+\omega x^{p^m}+\omega^2 x^{2p^m} +...+\omega^{p-1} x^{(p-1)p^m})=\sum_{n=0}^{p^{M+1}-1} \omega^{v_p(n)}x^{n}
\end{equation}
For $p=2$, equation (\ref{eq:pgf-root-of-unity}) reduces to the classical product generating function formula for the PTM sequence $v(n)$ (see \cite{AS},\cite{B}):
\begin{equation} \label{eq:product-generating-function}
 \prod_{m=0}^{N} (1-x^{2^m}) =\sum_{n=0}^{2^{N+1}-1} (-1)^{v(n)}x^n 
\end{equation}

In this paper, we present a new proof of Prouhet's solution by generalizing (\ref{eq:pgf-root-of-unity}) to polynomials whose coefficients sum to zero while preserving the form of (\ref{eq:product-generating-function}).  This was achieved by observing that the key ingredient in the proof of (\ref{eq:pgf-root-of-unity}) relies on the property that all $p$-th roots of unity sum to zero, namely,
\[
\omega^0+\omega^1+...+\omega^{p-1}=0
\]
where $\omega$ is a primitive $p$-th root of unity.  Towards this end, let $A=(a_0,a_1,...,a_{p-1})$ be a vector consisting of $p$ arbitrary complex values that sum to zero:
\[
a_0+a_1+...+a_{p-1}=0
\]
We define $F_N(x;A)$ to be the polynomial of degree $p^N-1$ whose coefficients belong to $A$ and repeat according to $v_p(n)$, i.e.
\begin{equation} \label{de:F}
F_N(x;A)=\sum_{n=0}^{p^{N}-1} a_{v_p(n)}x^n
\end{equation}
We then prove in Theorem \ref{th:factor-F} that for every positive integer $N$, there exists a polynomial $P_N(x)$ such that
\begin{equation} \label{eq:pgf-mod-p}
F_N(x;A)=P_N(x)\prod_{m=0}^{N-1}(1-x^{p^m})
\end{equation}
For example, if $p=3$ so that $a_0+a_1+a_2=0$, then (\ref{eq:pgf-mod-p}) becomes
\[
a_0+a_1x+a_2x^2 =[a_0+(a_0+a_1)x](1-x)
\]
and
\[
a_0+a_1x+a_2x^2+a_1x^3+a_2x^4+a_0x^5+a_2x^6+a_0x^7+a_1x^8 = [a_0+(a_0+a_1)x+(a_0+a_1)x^3+a_1x^4](1-x)(1-x^3)
\]
for $N=1$ and $N=2$, respectively.  In the case where $p=2$, $a_0=1$, and $a_1=-1$, then $P_N(x)=1$ for all $N$ and therefore (\ref{eq:pgf-mod-p}) reduces to (\ref{eq:product-generating-function}).  

Equation (\ref{eq:pgf-mod-p}) is useful in that it allows us to establish that the polynomial $F_N(x,A)$ has a zero of order $N$ at $x=1$ from which Prouhet's solution follows easily by setting $N=M+1$ and differentiating $F_N(x;A)$ $m$ times as demonstrated in Theorem \ref{th:pte}.

\section{Proof of the Prouhet-Tarry-Escott Problem}

Let $p\geq 2$ be a fixed integer.  We begin with a lemma that describes a recurrence for $F_N(x;A)$ whose proof follows from the fact that
\[
v_p(n+kp^m)=(v_p(n)+k)_p \ \ \ \ (0\leq n <p^m, 0\leq k < p)
\]
where we define $(n)_p = n$ mod $p$.  Moreover, let $A_k$ denote the $k$-th left cyclic shift of the elements of $A$, i.e.
\[
A_k=(a_{(k)_p},a_{(k+1)_p},...,a_{(p-1+k)_p})
\]

\begin{lemma} For any integer $N>1$, we have
\begin{equation} \label{eq:recurrence-F}
F_N(x;A)=F_{N-1}(x;A_0)+x^{p^{N-1}}F_{N-1}(x;A_1)+...+x^{(p-1)p^{N-1}}F_{N-1}(x;A_{p-1})
\end{equation}
\end{lemma}

\begin{proof} We have
\begin{align*}
F_{N}(x;A) & =\sum_{n=0}^{p^{N}-1} a_{v_p(n)}x^n \\
& = \sum_{n=0}^{p^{N-1}-1} a_{v_p(n)}x^n +\sum_{n=p^{N-1}}^{2p^{N-1}-1} a_{v_p(n)}x^n+...+\sum_{n=(p-1)p^{N-1}}^{p^{N}-1} a_{v_p(n)}x^n \\
& = \sum_{n=0}^{p^{N-1}-1} a_{v_p(n)}x^n +x^{p^{N-1}}\sum_{n=0}^{p^{N-1}-1} a_{v_p(n+p^{N-1})}x^n+...+x^{(p-1)p^{N-1}}\sum_{n=0}^{p^{N-1}-1} a_{v_p(n+(p-1)p^{N-1})}x^n \\
& = \sum_{n=0}^{p^{N-1}-1} a_{v_p(n)}x^n +x^{p^{N-1}}\sum_{n=0}^{p^{N-1}-1} a_{(v_p(n)+1)_p}x^n+...+x^{(p-1)p^{N-1}}\sum_{n=0}^{p^{N-1}-1} a_{(v_p(n)+p-1)_p}x^n \\
& = F_{N-1}(x;A_0)+x^{p^{N-1}}F_{N-1}(x;A_1)+...+x^{(p-1)p^{N-1}}F_{N-1}(x;A_{p-1})
\end{align*}
\end{proof}

\noindent For example, let $p=3$ and $A=(a_0,a_1,a_2)$.  Then
\begin{align*}
F_1(x;A) & = a_0+a_1x +a_2x^2\\
F_2(x;A) & = a_0+a_1x+a_2x^2+a_1x^3+a_2x^4+a_0x^5+a_2x^6+a_0x^7+a_1x^8 \\
& = F_1(x;A_0)+x^3F_1(x;A_1)+x^6F_1(x;A_2)
\end{align*}

Next, define a recursive sequence of vectors consisting of unknown constants as follows:
\[
C_1=(c_{0},...,c_{p-2})
\]
and for $N>1$, 
\begin{equation} \label{eq:recurrence-C}
C_N  = C_{N-1}(0)\# C_{N-1}(1)\# ... \# C_{N-1}(p-2) 
\end{equation}
where $\#$ denotes concatenation of vectors and
\[
C_{N-1}(k)=(c_{j+kp^{N-1}}:c_j\in C_{N-1})
\]
for $k=0,1,...,p-2$.  For example, if $p=3$, then
\begin{align*}
C_1 & =(c_0,c_1) \\
C_2 & =C_1(0)\#C_1(1)=(c_0,c_1,c_3,c_4) \\
C_3 & =C_2(0)\#C_2(1)=(c_0,c_1,c_3,c_4,c_9,c_{10},c_{12},c_{13})
\end{align*}
Note that if $p=2$, then $C_N=(c_0)$ for all $N\geq 1$.

Moreover, define a sequence of polynomials $P_N(x;C_N)$ recursively as follows:
\[
P_1(x;C_1)=c_0+c_1x+...+c_{p-2}x^{p-2}
\]
and for $N>1$,
\begin{equation} \label{eq:recurrence-P}
P_N(x;C_N)= P_{N-1}(x;C_{N-1}(0))+x^{p^{N-1}}P_{N-1}(x;C_{N-1}(1))+ ... +x^{(p-2)p^{N-1}}P_{N-1}(x;C_{N-1}(p-2)) 
\end{equation}

We are now ready to prove that $F_N(x;A)$ has the following factorization.

\begin{theorem} \label{th:factor-F}
Let $N$ be a positive integer.  There exists a polynomial $P_N(x;C_N)$ such that
\begin{equation} \label{eq:ptm-polynomial}
F_N(x;A)=P_{N}(x;C_N)\prod_{m=0}^{N-1}(1-x^{p^m})
\end{equation}
\end{theorem}

\begin{proof} We prove (\ref{eq:ptm-polynomial}) by induction.  First, define $Q_N(x)=\prod_{m=0}^{N-1}(1-x^{p^m})$ so that for $N>1$,
\begin{equation} \label{eq:recurrence-Q}
Q_N(x)=Q_{N-1}(x)(1-x^{p^{N-1}})
\end{equation}
To establish the base case $N=1$, we expand $F_1(x;A)=P_{1}(x;C_1)Q_1(x)$ to obtain
\[
a_0+a_1x+...+a_{p-1}x^{p-1}=c_0+(c_1-c_0)x+...+(c_{p-2}-c_{p-1})x^{p-2}-c_{p-2}x^{p-1}
\]
Then equating coefficients yields the system of equations
\begin{align*}
c_0 & =a_0 \\
c_1-c_{0}  & =a_1 \\
& \ ... \\
c_{p-2}-c_{p-1} & = a_{p-2} \\
-c_{p-2} & = a_{p-1}
\end{align*}
Since $a_0+a_1+...+a_{p-1}=0$, this system is consistent with solution $c_m=\sum_{k=0}^m a_m$ for $m=0,1,...,p-2$ where $c_{p-2}=a_0+...+a_{p-2}=-a_{p-1}$.  Thus, $P_1(x;C_1)$ is given by
\[
P_{1}(x;C_1)=\sum_{m=0}^{p-2}\left(\sum_{k=0}^m a_k\right)x^m
\]
Note that if $p=2$, then $P_1(x;C_1)=a_0$.

Next, assume there exists a polynomial $P_{N-1}(x;C_{N-1})$ that solves
\[
F_{N-1}(x;A)=P_{N-1}(x;C_{N-1})Q_{N-1}(x)
\]
To prove that there exists a solution $P_N(x;C_N)$ for
\begin{equation} \label{eq:FPQ}
F_N(x;A)=P_N(x;C_N)Q_N(x)
\end{equation}
we expand (\ref{eq:FPQ}) using recurrences (\ref{eq:recurrence-F}), (\ref{eq:recurrence-P}), and (\ref{eq:recurrence-Q}):
\begin{equation} \label{eq:FPQ-expand}
\sum_{k=0}^{p-1}x^{kp^{N-1}}F_{N-1}(x;A_{p-1})=\left[\sum_{k=0}^{p-2}x^{kp^{N-1}}P_{N-1}(x;C_{N-1}(k)) \right]Q_{N-1}(x)(1-x^{p^{N-1}})
\end{equation}
We then equate coefficients in (\ref{eq:FPQ-expand}) corresponding to the terms $x^{kp^{N-1}}$.  This yields the system of equations
\begin{align*}
F_{N-1}(x;A_0) & = P_{N-1}(x;C_{N-1}(0))Q_{N-1}(x) \\
F_{N-1}(x;A_1) & = [P_{N-1}(x;C_{N-1}(1))-P_{N-1}(x;C_{N-1}(0))]Q_{N-1}(x) \\
& \ ... \\
F_{N-1}(x;A_{p-2}) & = [P_{N-1}(x;C_{N-1}(p-2))-P_{N-1}(x;C_{N-1}(p-3))]Q_{N-1}(x) \\
F_{N-1}(x;A_{p-1}) & = -P_{N-1}(x;C_{N-1}(p-2))Q_{N-1}(x)
\end{align*}
Now, each equation above corresponding to $F_{N-1}(x;A_k)$ for $k=1,...,p-2$ can be replaced by one obtained by summing all equations up to $k$, namely
\[
F_{N-1}(x;B_k)  = P_{N-1}(x;C_{N-1}(k))Q_{N-1}(x)
\]  
where $B_k=A_0+...+A_k$ is defined by vector summation.  This yields the equivalent system of equations
\begin{align*}
F_{N-1}(x;B_0) & = P_{N-1}(x;C_{N-1}(0))Q_{N-1}(x) \\
F_{N-1}(x;B_1) & = P_{N-1}(x;C_{N-1}(1))Q_{N-1}(x) \\
& \ ... \\
F_{N-1}(x;B_{p-2}) & = P_{N-1}(x;C_{N-1}(p-2))Q_{N-1}(x) \\
F_{N-1}(x;A_{p-1}) & = -P_{N-1}(x;C_{N-1}(p-2))Q_{N-1}(x)
\end{align*}
By induction, each of the equations above corresponding to $F_{N-1}(x;B_k)$ has a solution in $C_{N-1}(k)$.  Moreover, the last equation corresponding to $F_{N-1}(x;A_{p-1})$ is equivalent to the equation corresponding to $F_{N-1}(x;B_{p-2})$ since $B_{p-2}=A_0+...+A_{p-2}=-A_{p-1}$.  This proves that  (\ref{eq:FPQ}) has a solution in $C_N$ because of (\ref{eq:recurrence-C}).
\end{proof}

We now present our proof of the Prouhet-Tarry-Escott problem.

\begin{theorem}[\cite{L},\cite{P}-\cite{W2}] \label{th:pte}
Let $M$ be a positive integer, $L=p^{M+1}$, and $S_0$, $S_1$,...,$S_{p-1}$ a partition of $\{0,1,...,L-1\}$ defined by
\[
n\in S_{v_p(n)}
\]  
for $0\leq n \leq L-1$.  Then $S_0$, $S_1$,...,$S_{p-1}$ have equal sums of like powers of degree $M$, i.e.
\[
\sum_{n\in S_0} n^m =\sum_{n\in S_1} n^m= ...=\sum_{n\in S_{p-1}} n^m 
\]
for all $m=0,1,...,M$. 
\end{theorem}

\begin{proof} Denote by $s_k(m)=\sum_{n\in S_{k}} n^m$. Let $A=(a_0, a_1,...,a_p-1)$ be a vector of $p$ arbitrary complex values that sum to zero: $a_0+a_1+...+a_{p-1}=0$.  Set $N=M+1$ and define $F_{N}(x;A)$ as in (\ref{de:F}).  Next, substitute $x=e^{\theta}$ into $F_N(x;A)$ and compute the $m$-th derivative of $G_{N}(\theta):=F_{N}(e^{\theta};A)$ at $\theta=0$.  Then on the one hand, we have from the standard rules of differentiation that
\begin{align*}
G^{(m)}_{N}(0) & =\sum_{n=0}^{p^N-1}n^m a_{v_p(n)} \\
& = \sum_{n\in S_0}n^m a_0 +...+ \sum_{n\in S_{p-1}}n^m a_{p-1} \\
& = a_0s_0(m)+...+a_{p-1}s_{p-1}(m)
\end{align*}
On the other hand, we have from (\ref{eq:ptm-polynomial}) that $G_{N}(\theta)$ has a zero of order $N$ at $\theta=0$. It follows that
\[
G^{(m)}_{N}(0)=0
\]
for $m=0,1,...,N-1$.  Thus, 
\begin{equation} \label{eq:linear-equation}
a_0s_0(m)+...+a_{p-1}s_{p-1}(m)=0
\end{equation}
Now, recall that the values $a_0$, $a_1$, ..., $a_{p-1}$ can be chosen arbitrarily as long as they sum to zero.  Therefore, we choose them as follows: for any pair of distinct non-negative integers $j,k$ satisfying $0\leq j,k \leq p-1$, set $a_j=1$, $a_k=-1$, and $a_l=0$ for all $l\neq j,k$. Then (\ref{eq:linear-equation}) reduces to
\[
s_j(m)-s_k(m)=0
\]
or equivalently, $s_j(m)=s_k(m)$.  But since this holds for all distinct $j,k$, we conclude that
\[
s_0(m)=s_1(m)=...=s_{p-1}(m)
\]
for $m=0,1,...,M$ as desired.
\end{proof}

We conclude by explaining our motivation for studying the polynomials $F_N(x;A)$.  In \cite{NC}, it was shown that these polynomials arise in radar as ambiguity functions of pulse trains generated by complementary codes that repeat according to the Prouhet-Thue-Morse sequence.  Prouhet's solution was then used to demonstrate that these pulse trains, called complementary PTM pulse trains, are tolerant of Doppler shifts due to a moving target by establishing that their Taylor series coefficients vanish up to order $M$.

\end{document}